\documentclass{amsart}

\usepackage{amssymb}
\usepackage{bm}
\usepackage{graphicx}
\usepackage[centertags]{amsmath}
\usepackage{amsfonts}
\usepackage{amsthm}

\newtheorem{theorem}{Theorem}
\newtheorem*{mtheo}{Theorem}
\newtheorem{coroll}{Corollary}
\newtheorem{claim}{Claim}
\newtheorem{lemma}{Lemma}
\newtheorem{proposition}{Proposition}
\newtheorem{definition}{Definition}

\newtheorem{observation}{Observation}
\theoremstyle{definition}

\begin{document}
 
\title[Canonical equivalence relations on fronts on $FIN_k$]{Canonical equivalence relations on fronts on $FIN_k$}

\author{ Dimitris VLITAS}

\address{Department of Mathematics, University of Toronto, 40 St. George Street, Toronto, Ontario, Canada M5S 2E4}
\email{vlitas@math.univ-paris-diderot.fr}

\begin{abstract}  We prove that for every equivalence relateion on a barrier  on the space $\langle FIN^{[\infty]}_k,\leq,r\rangle$, for any $k$, there exists $Y\in FIN_k^{[\infty]}$ so that the restriction of the coloring on $\langle Y\rangle$ is canonical.
\end{abstract}
\maketitle

\section{Introduction}

 Canonical results in Ramsey theory try to  describe equivalence relations in a given Ramsey structure,
based on the underlying pigeonhole principles. The first example of them is the classical Canonization
Theorem  by P. Erd\H{o}s and R. Rado \cite{Er-Ra} which can be presented as follows: Given $\alpha\le \beta
\le \omega$ let
$$\binom{\beta}{\alpha}:=\{f (\alpha)  \,:\, f:\alpha \rightarrow \beta \text{ is strictly increasing} \}.$$
The previous is commonly denoted by $[\beta]^{\alpha}$.   Then for any $n<\omega$  and any finite coloring of
$\binom{\omega}{n}$ there is an isomorphic copy $M$ of $\omega$ (i.e. the image of an strictly increasing
$f:\omega\rightarrow\omega$) and some
   $I\subseteq n(:=\{0,1,\dots,n-1\})$ such that any two  $n$-element subsets have the same color  if and only if they agree
  on the corresponding relative positions given by $I$.

This   was extended by P. Pudl\'ak and V. R\"odl in \cite{Pu-Ro}  for colorings of  a given \emph{uniform}
family $\mathcal{G}$ of finite subsets of $\omega$ by showing that given any
coloring of  $\mathcal{G}$ of finite subsets of $\omega$, there exists $A$ an infinite subset of $\omega$, a
uniform family $\mathcal{T}$ and a mapping $f:\mathcal{G}\to \mathcal{T}$ such that $f(X)\subseteq X$ for all
$X\in \mathcal{G}$ and such that any two $X,Y\in \mathcal{G}\upharpoonright A$ have the same color   if and
only if $f(X)=f(Y)$. Since then, many results of similar nature have been obtained (see \cite{Mil}, \cite{Vli}).

In 1992 G.T.Gowers \cite{Gow} obtained a stability result for real valued Lipschitz functions defined in the unit sphere $c_0$. This result is actually a consequence of a deep infinite dimensional Ramsey type result, which gives rise to a Ramsey space. In this paper we canonize equivalence relations on that space. To state our result we need to introduce some notions. 

 Given a positive integer $k$, let $FIN_k$ be the set of mappings $x:\omega \to \{0,1,\dots, k\}$, called $k$-vectors, whose support $suppx=\{n:x(n)\neq 0\}$ is finite and with $k$ in their range. An element $X$ of $FIN_k^{[\infty]}$ is a sequence $X=(x_n)_{n\in \omega}$ so that $\max supp x_n< \min supp x_{n+1}$, for all $n\in \omega$. We write $x_n<x_{n+1}$ to show that $\max supp x_n< \min supp x_{n+1}$.
 
 Let $T: FIN_k\to FIN_{k-1}$ be the map defined by $T(x)(n)=\max\{ x(n)-1,0\}$.
The $k$-combinatorial subspace $\langle X\rangle$ is the set of combinations of the form: $$T^{i_0}x_{n_0}+\dots + T^{i_m}x_{n_m}$$ with the condition that at least one $i_j=0$, $j<m+1$. By $T^i$ we mean $T^ix(n)=\max( x(n)-i,0)$, for $i>0$ and $T^0=id$. Given $X=(x_n)_{n\in l}$, we define the {\it{length}} of $X$, denoted by $|X|$, to be equal to $l$.
For $X=(x_n)_{n\in |X|}$, $Y=(y_n)_{n\in |Y|}\in FIN_k^{[\leq \infty]}$, set $X\leq Y$ if $x_n\in \langle Y\rangle$ for all $n<|X|$. In this case we say that $X$ is a $\it{ block-subsequence }$ of $Y$. Then $\leq$ is a partial ordering on $FIN_k^{[\infty]}$. 

For $X\in FIN^{[\infty]}_k$ we define $r_n(X)=(x_i)_{i\in n}$. We set $\mathcal{A}X_n=\{ r_n(Y): Y\leq X\}$ and $\mathcal{A}X=\cup_{n\in \omega} \mathcal{A}X_n$.
Next we introduce a topology on $ FIN^{[\infty]}_k$ with basic open sets as follows. For $s\in \mathcal{A}X_n$, by $|s|=n$ we denote is length. We define $$[s,X]=\{Y:Y\leq X, r_n(Y)=s\}.$$ These set form the basic sets for a topology on the space $FIN^{[\infty]}_k$.

 In \cite{To} it is shown that $\langle FIN^{[\infty]}_k,\leq,r\rangle$ satisfies axioms $A.1-A.4$, so it is a topological Ramsey space. Two corollaries of being such a space are the following.
 
 \begin{coroll}Let $c:\mathcal{A}X_n\to l$ be a finite coloring. There exists an $Y\leq X$ so that $c\upharpoonright \mathcal{A}Y_n$ is constant. \end{coroll}
 and 
 \begin{coroll}Let $s\in \mathcal{A}X$ and $c:[s,X]\to l$ be a finite Suslin measurable coloring. There exists $Y\leq [s,X]$ so that $c\upharpoonright [s,Y]$ is constant.
 \end{coroll}
 Recall that a map $f:X \to Y$ between two topological spaces is Suslin measurable, if the preimage $f^{-1}(U)$ of every open subset $U$ of $Y$ belong to the minimal $\sigma-$field of subsets of $X$ that contains its closed sets and it is closed under the Suslin operation \cite{Ke}. 
 
For a family of finite approximation of elements of $ FIN^{[\infty]}_k$ called fronts (see Definition $1$), we prove the following.

 \begin{mtheo} Let $f: \mathcal{F}\to \omega$ be a coloring of a front on $[\emptyset, X]$, for $X\in FIN_k^{[\infty]}$. There exists $Y_0\leq X$ so that $f\upharpoonright \mathcal{F}\upharpoonright \langle Y_0 \rangle$ is canonical.
 \end{mtheo}
 
 Here by canonical, we mean that there exists a map $\phi$ so that for every $(t_0, \dots, t_{d-1})\in \mathcal{F}\upharpoonright \langle Y_0 \rangle$ $$\phi(t_0,\dots, t_{d-1})\subseteq (T^{i_0}t_{i_0}+\dots + T^{i_l}t_{i_l}, \dots,T^{j_0}t_{j_0}+ \dots + T^{j_m}t_{j_m} ),$$ where $\{i_0,\dots, i_l, \dots, j_0,\dots, j_m\}\subseteq  d$ and every $t_i$, $i<d$ appears in at most one combination in $\{T^{i_0}t_{i_0}+\dots + T^{i_l}t_{i_l}, \dots,T^{j_0}t_{j_0}+ \dots + T^{j_m}t_{j_m} \}\subset \langle X \rangle$. Then $\phi$ is so that for any $s,t\in \mathcal{F}$ it holds $f(s)=f(t)$ if and only if $\phi(s)=\phi(t)$. We view elements on $FIN_k$, the $k$-vectors, in the set theorytic way, as sets of ordered pairs. The subset is taken in this sense.
 The proof of the above theorem is divided in two parts. In the next section we add the necessary definitions and new concepts and we present the first part of the proof. In the final section we present the second part.

    \section{ Main theorem }

The above partial ordering $\leq $ on $FIN_k^{[\infty]}$, allows the finitization $\leq_{fin}$ as follows: for $X=(x_i)_{i\in h}$, $Y=(y_j)_{j\in m}$, we say that $X\leq_{fin} Y$ if and only if $X\leq Y$ and $(\forall l<m), X\nleq,Y\upharpoonright l$.

For $s\in \mathcal{A}X$ and $X\in FIN_k^{[\infty]}$ we define the depth of s in X as follows:
\begin{equation*}
depth_X(s)=
\begin{cases}
\min\{k: s \leq_{fin} r_k(X) \} & \text{ if } (\exists k) s \leq_{fin} r_k(X),\\
\infty & \text{otherwise.}
\end{cases}
\end{equation*}

Given an non empty basic open set $[s,X]$, $|s|=n$, let $$[s,X]_{n+1}=\{t\in \mathcal{A}X_{n+1}: s\sqsubseteq t \}.$$
Now we introduce the notion of a $\it{Front}$.

\begin{definition} A family $\mathcal{F}$ of finite approximations of reducts of $X$ is called a front, if for every $Y\leq X$, there exists $s\in \mathcal{F}$ so that $s=r_n(Y)$ and for any two distinct $s,t\in \mathcal{F}$, is not the case that $s\sqsubseteq t$.
\end{definition}
 
 We distinguish specific instances of fronts on $X$, the $\mathcal{A}X_n$. 
 Given a front $\mathcal{F}$ on $[\emptyset, X]$, we introduce $\hat{\mathcal{F}}$ defined as follows: $$\hat{\mathcal{F}}=\{ t\in \mathcal{A}X: \exists s\in \mathcal{F}, t\sqsubseteq s\}$$ observe that $\emptyset \in \hat{\mathcal{F}}$. For $t\in \hat{\mathcal{F}}\setminus \mathcal{F}$ $$\mathcal{F}_t=\{s\in \mathcal{F}: t\sqsubseteq s \}.$$
 
 For $Y\leq X$ $\mathcal{F}\upharpoonright  Y  =\{t\in \mathcal{F}: t\in \mathcal{A}Y', Y'\leq Y\},$
 
 $$\hat{\mathcal{F}}\upharpoonright Y  =\{t\in \hat{\mathcal{F}}: t\in \mathcal{A}Y', Y'\leq Y\}.$$
 
  Finally for $s\in  \mathcal{A}X$ and $X\in FIN_k^{[\infty]}$ by $X/s$ we denote $X\setminus s$. Similarly for $s,t\in  \mathcal{A}X$, by $X/(s,t)$ we denote $X\setminus s \cap X\setminus t$.

 The following proposition is a fact that holds in any topological Ramsey space. For the shake of completeness, we give a proof here in the context of our space.

\begin{proposition} Suppose the $\langle  FIN_k^{[\infty]}, \leq, r\rangle$ has the property that given a property $\mathcal{P}(\cdot, \cdot)$, $s\in \mathcal{A}X$ and $Y\leq X$, there exists $Z' \leq Y$ so that $\mathcal{P}(s,Z' )$. Then there exists $Z\leq X$ such that for any $s\in \mathcal{A}Z$ it holds that $\mathcal{P}(s,Z)$.

Similarly for properties of the form $\mathcal{P}(\cdot,\cdot,\cdot)$. If given $s,t\in \mathcal{A}X$ and $Y\leq X$, there exists $Z'\leq Y$ so that $\mathcal{P}(s,t,Z')$. Then there exists $Z\leq X$ so that $\mathcal{P}(s,t,Z)$ for all $s,t\in \mathcal{A}Z$.
\end{proposition}

\begin{proof} 

Let $t_0=r_0(X)$ and $X$. There exists $X_0\leq X$ so that $\mathcal{P}(t_0,X_0)$. Set $t_1=r_1(X_0)$ and let $X_1\leq X_0$ so that $\mathcal{P}(t_1,X_1)$ holds. Set $t_2=r_2(X_1)$. Consider the finite set $A_2=\{ z \in \mathcal{A}X: z\leq_{fin} t_2 \}$. For every $z\in A_2$ there exists $Y\leq X_0$ so that $\mathcal{P}(z,Y)$. After considering all $z\in A_2$ we get $X_2\leq X_1$ and $t_3=r_3(X_2)$ so that $\mathcal{P}(z ,X_2)$ holds, for all $z\in A_2$. Suppose we have constructed $t_n$ and $X_n$. Set $t_{n+1}=r_{n+1}(X_n)$. Consider $A_{n}=\{z \in \mathcal{A}X: z\leq_{fin} t_{n+1}\}$. For every $z \in A_{n}$ there exists $Y\leq X_n$ so that $\mathcal{P}(z,Y)$. Therefore we get $X_{n+1}\leq X_n$ so that for any $z \in A_{n}$ we have $\mathcal{P}(z,X_{n+1})$. Set $t_{n+2}=r_{n+2}(X_{n+1})$. Proceed in that manner. Observe that for all $n\in \omega$ $t_n\sqsubset t_{n+1}$. Set $Z=\cup_{n\in \omega}t_n$. \\

 Now we prove similarly the second statement of our proposition. Let $t_0=r_0(X)$ and $t_1=r_1(X)$ and $X$. There exists $X_1\leq X$ so that $\mathcal{P}(t_0,t_1,X_1)$. Let $t_2=r_2(X_1)$. Consider the finite set $A_2=\{ z\in \mathcal{A}X: z\leq_{fin} t_2\}$. For any $(s,t)\in [A_2]^2$, there exists $Y\leq X_1$ so that $\mathcal{P}(s,t,Y)$. By exhausting all possible such a pairs we get $X_2\leq X_1$ such that for any $(s,t)\in [A_2]^2$ it holds that $\mathcal{P}(s,t,X_2)$. Set $t_3=r_3(X_2)$. Suppose we have constructed $t_n$ and $X_n$. Let $t_{n+1}=r_{n+1}(X_n)$ and $A_{n+1}=\{z\in \mathcal{A}X: z \leq_{fin} t_{n+1}\}$. For any pair $(s,t)\in [A_n]^2$ there exists $Y\leq X_n$ so that $\mathcal{P}(s,t,Y)$ holds. After considering all possible such a pairs, we get $X_{n+1}$ such that for any $(s,t)\in [A_{n+1}]^2$ it holds that $\mathcal{P}(s,t,X_{n+1})$. Set $t_{n+2}=r_{n+2}(X_{n+1})$. Observe that for every $n\in \omega$, $t_{n}\sqsubset t_{n+1}$. Let $Z=\cup_{n\in \omega} t_n$. 
\end{proof}

As mentioned above in \cite{To} is shown that $\langle FIN^{[\infty]}_k,\leq,r\rangle$ satisfies a pigeon hole property (axiom $A.4$ in \cite{To}).  
 To show that this space satisfies a strengthen pigeon hole property, (Theorem $1$ below), we introduce some concepts and definitions from \cite{Ab}.
We restrict our attention to the subset of $FIN_k$ that contains all the $k$-vectors that are system of staircases.

\begin{definition}Given an integer $i\in [1,k]$, let $\min_i,\max_i:FIN_k\to \omega$ be the mappings $\min_i x=\min_i x^{-1}(\{i\})$, $\max_i x=\max x^{-1}(\{i\})$, if defined and $0$ otherwise. A $k$-vector $x$ is a system of staircases (\it{sos} in short) if and only if
\begin{enumerate}
\item[1] Range $x=\{0,1,\dots ,k\}$,
\item[2] $\min_ix<\min_j x<\max_i x$, for $i<j\leq k$,
\item[3]  for every $1\leq i\leq k$,

\end{enumerate}
\begin{enumerate}
\item[] Range $x\upharpoonright [\min_{i-1}x,\min_i x)=\{0,\dots, i-1\}$,
\item[] Range $x\upharpoonright (\max_{i} x,\max_{i-1}x]=\{ 0,\dots, i-1\}$,
\item[] Range $x\upharpoonright [\min_k x, \max_k x]=\{0,\dots,k\}$.

\end{enumerate}

\end{definition}

An $X=(x_n)\in FIN_k^{[\infty]}$ is a system of staircases if and only if every $k$-vector $x_n$ is an sos. In \cite{Ab} it is shown that for every $X\in FIN_k^{[\infty]}$ there exists $Y\leq X$ so that $X$ is sos and that $T$ preserves sos. Next we introduce some mappings.

\begin{definition}
Let $X\in FIN_k^{[\infty]}$ be a sos. The mapping $\min_i$, for $i\in [1,k]$, can be interpreted as
 \begin{equation*}
\min_i(w) (n)=
\begin{cases}
i & \text{ if } n=\min_i(w),\\
0 & \text{ otherwise.}
\end{cases}
\end{equation*}
for $w\in \langle X\rangle$. Extending this, define for $I\subseteq \{1,\dots k\}$, the mapping $\min_I:\langle A\rangle_k\to FIN_{\max I}\subseteq FIN_{\leq k}$ by $\min_I(w)(n)=i$ if $n=\min_i(w)$, for $i\in I$ and $0$ otherwise, i.e. $\min_I(w)=\{ (\min_i(w),i): i\in I\}$, and extended by $0$. Similarly, let 
 \begin{equation*}
\max_i(w) (n)=
\begin{cases}
i & \text{ if } n=\max_i(w),\\
0 & \text{ otherwise.}
\end{cases}
\end{equation*}
and let $\max_I:FIN_k\to FIN_{\max_I}$ be defined by $\max_I(w)=\{ (\max_i(s),i):i\in I\}$, again extended by $0$. Clearly $\min_I=\bigvee_{i\in I} \min_i$ and $\max_I=\bigvee_{i\in I}\max_i$, where for two mappings $f,g:\langle X\rangle_k\to FIN_{\leq k}$ we define $(f\vee g)(w)=f(w)\vee g(w)$.

\end{definition}

For $l\leq i$, let $\theta^0_{i,l}:\langle X\rangle \to FIL_l$ be the mappings defined by 
\begin{enumerate}
\item[] $\theta^0_{i,l}(w)=\{(n,l):n\in (\min_{i}(w),\min_{i+1}(w)), w(n)=l\}$ extended by $0$,
\item[] $\theta^1_{i,l}(w)=\{(n,l):n\in (\max_{i+1}(w),\max_i(w)), w(n)=l\}$ extended by $0$,
\item[] $\theta^2_{l}(w)=\{(n,l):n\in (\min_k(w),\max_k(w)),w(n)=l\}$ extended by $0$, where $l\leq k$.
\end{enumerate}

\begin{definition}Let $\mathcal{G}_{\min}=\{\min_1,\dots, \min_k\}$, $\mathcal{G}_{\max}=\{ \max_1,\dots, \max_k\}$, $\mathcal{G}_{mid^\epsilon}=\{\theta^\epsilon_{i,l}:i\in \{1,\dots,k-1\}, l\in\{1,\dots,i-1\}\}$, for $\epsilon=0,1$ and $\mathcal{G}_{mid}=\{ \theta^2_l:l\in \{1,\dots, k\}\cup \{0\}$. Set $\mathcal{G}=\mathcal{G}_{\min}\cup \mathcal{G}_{\max} \cup \mathcal{G}_{mid^0} \cup \mathcal{G}_{mid^1} \cup \mathcal{G}_{mid}$.
\end{definition}

Given a $k$-block sequence $X$, we say that $f:\langle X \rangle \to FIN_{\leq k}$ is a staircase function if it is in the lattice closure of $\mathcal{G}$. An equivalence relation $R$ on $\langle X \rangle$ is a staircase relation if the following holds: $s R t$ if and only if $f(s)=f(t)$, for some staircase mapping $f$. In \cite{Ab} is shown that if $f$ is a staircase map then there are $I_{\epsilon}\subseteq \{1,\dots,k\}, J_{\epsilon}\subseteq \{j\in I_{\epsilon}:j-1\in I_{\epsilon}\}$, $(l_j^{(\epsilon)})_{j\in J_\epsilon}$ with $l_j^{(\epsilon)}\leq j-1$, for $\epsilon=0,1$ and $l_k^{(2)}$ such that 

$$f=:\min_{I_0}\vee \bigvee_{j\in J_0}\theta^0_{j-1,l_j^{(0)}}\vee \theta^2_{l_k^{(2)}}\vee \max_{I_1}\vee \bigvee_{j\in J_1}\theta^1_{j-1,l_j^{(1)}}.$$

As with the $k$-vectors, given two functions $f,g:\langle A\rangle \to FIN_{\leq k}$, we write $f<g$ to denote that $f(w)<g(w)$, for all $w\in \langle A\rangle$. Therefore any staircase mapping $f$ has a unique decomposition $f=\cup_{i\in n}f_i$ with $f_0<f_1<\dots <f_{n-1}$ in $\mathcal{G}$.

In \cite{Ab} J. Lopez-Abad showed the following.

\begin{theorem}For every $k$ and every equivalence relation on $FIN_k$ there is a system of staircases $B$ such that the equivalence relation restricted to $\langle B\rangle$ is a staircase equivalence relation.
\end{theorem}

 Observe that the above theorem is the one dimentional case of our main theorem mentioned in the introduction. In other words it takes care of the front $\mathcal{A}X_1$, for $X\in FIN_k^{[\infty]}$.
  Next we make the following definition.
  
  \begin{definition}Given $\mathcal{F}$ a front on $[\emptyset, X]$ and $f:\mathcal{F}\to \omega$. Fix $s,t\in \hat{\mathcal{F}}\setminus \mathcal{F}$ and $X$. $X$ separates $s$ and $t$ if and only if for all $w\in \langle X/s \rangle $ and $v\in\langle  X/t\rangle $ so that $s\cup w,t\cup v\in \mathcal{F}\upharpoonright X$, $f(s\cup w)\neq f(t\cup v)$. $X$ mixes $s$ with $t$, if there is no $Y\leq X$ which separates $s$ with $t$. $X$ decides for $s$ with $t$ if and only if either $X$ mixes $s$ with $t$, or $X$ separates $s$ with $t$.
\end{definition}

Therefore $X$ mixes $s$ with $t$ if and only if for each $Y\leq X$, there are $w\in \langle Y/s\rangle $ and $v \in \langle Y/t\rangle$ so that $s\cup w, t\cup v\in \mathcal{F}\upharpoonright Y$ and $f(s\cup w)=f(t\cup v)$.

The following proposition follows directly from the definitions.

\begin{proposition} The following hold.
\begin{enumerate}
\item{} If $X$ mixes (separates) $s$ with $t$, so does any reduct $Y\leq X$.
\item{} For every $s,t\in \hat{\mathcal{F}}\setminus \mathcal{F}$ if for any $w\in [s,X]_{|s|+1}$ there exists $v\in [t,X]_{|t|+1}$ so that $X$ mixes $s\cup w$ with $t\cup v$, then $X$ also mixes $s$ with $t$.
\end{enumerate}

\end{proposition}

Next we observe the following.

\begin{proposition}Given $X$ and a front $\mathcal{F}$ on $[\emptyset, X]$, there exists $Z\leq X$ so that for all $s,t\in\hat{ \mathcal{F}}\upharpoonright Z$, $Z$ decides $s$ with $t$.
\end{proposition}

\begin{proof} Given $s,t$ and $Y\leq X$ it suffices to show that there exists $Z\leq Y$ which decides for $s$ and $t$. Then the statement of the this proposition will follow from Proposition $1$ and the property $\mathcal{P}(s,t,Y)$ stating that $Y$ decides for $s$ and $t$. Assume that $depth_X(s) \leq depth_X(t)$ and consider the two-coloring: $c':[t, Y]\to 2$ defined by 
\begin{equation*}
c'(Y')=
\begin{cases}
1 & \text{ if } \exists w\in\langle  Y'\rangle , v\in \langle Y' \rangle \text { so that }  Y' \text{ mixes } s\cup w \text{ with } t\cup v,\\
0 & \text{ otherwise.}
\end{cases}
\end{equation*}

Corollary $2$, provides us with $Z$ that either mixes $s$ with $t$, in the case that $c'\upharpoonright [t, Z]=1$, or separates them, in the case that $c'\upharpoonright [t, Z]=0$.

\end{proof}

 The above notion of mixing induced by Definition $5$ is not necessarily transitive. 
 An example of such an equivalence relation is the following. Let $X=(x_n)_{n\in \omega}\in FIN_1^{[\infty]}$ and $c:\mathcal{A}X_2\to \omega$, defined as follows $c(s,t)=s+t=s\cup t$. Let $s=x_0$, $t=x_0+x_2$ and $p=x_0+x_1+x_2$. Observe that the depth of $s,t$ and $p$ in $X$ is not the same. Then $X$ mixes $s$ with $t$ and $s$ with $p$, but $X$ does not  mixes $t$ with $p$. The above example generalizes easily for any $k$.

In the case of the same depth, mixing is transitive as the next lemma shows.

\begin{lemma}Let $s,t,p\in \mathcal{A}X_n$, with $depth_X(s)=depth_X(t)=depth_X(p)$. If $X$ mixes $s$ with $t$ and $X$ mixes $t$ with $p$, then $X$ mixes $s$ with $p$.
\end{lemma}

\begin{proof} Suppose that $X$ mixes $s$ with $t$ and $X$ mixes $t$ with $p$, but $X$ separates $s$ with $p$. Consider the two-coloring $c_1:[p, X]_{n+1} \to 2$ defined by 
\begin{equation*}
c_1(p\cup w)=
\begin{cases}
1 & \text{ if } \exists t\cup v \in [t,X]_{n+1}, \text{ and } $X$ \text{ mixes } p\cup w \text { with } t\cup v,\\
0 & \text{ otherwise.}
\end{cases}
\end{equation*}
Corollary $1$, gives us a $Y\in [p, X]$ so that $c_1\upharpoonright [p,Y]_{n+1} =1$. Similarly we consider the two-coloring $c_2:[t,Y]_{n+1}\to 2$ defined by:
\begin{equation*}
c_2(t\cup v)=
\begin{cases}
1 & \text{ if } \exists s\cup v'\in [s,Y]_{n+1}, \text{ and } $Y$ \text{ mixes } t\cup v \text { with } s\cup v',\\
0 & \text{ otherwise.}
\end{cases}
\end{equation*}

 which gives us a $Z\in [t, Y]$ so that $c_2\upharpoonright [t,Z]_{n+1}=1$. But this implies that $Z\leq X$ mixes $s$ with $p$, a contradiction.

\end{proof}

Now we are ready to state and prove our main theorem.
 
 \begin{theorem} Let $f: \mathcal{F}\to \omega$ be a coloring of a front on $ X$, for $X\in FIN_k^{[\infty]}$. There exists $Y_0\leq X$ so that $f\upharpoonright \mathcal{F}\upharpoonright  Y_0 $ is canonical.
 \end{theorem}

  Let $\mathcal{F}$ be any front on $\langle X\rangle$ and $f:\mathcal{F}\to \omega$ any coloring. Then Theorem $1$ looks after the case of front $\mathcal{F}=\mathcal{A}X_1$. Consider an arbitrary front $\mathcal{F}$. Definition $5$ gives us the notion of separation and mixing. We are going to divide the proof in two parts. In the first part we assume that mixing is transitive. In the second part we deal with the case that mixing is not transitive.  In the first case we are going to obtain a map $\phi$ , $\phi(t_0,\dots, t_{d-1})\subseteq (t_0, \dots, t_{d-1})$ so that for any $s,t\in \mathcal{F}$ it holds that $f(s)=f(t)$ if and only if $\phi(s)=\phi(t)$. In the second case we obtain a map $\phi$ so that $\phi(t_0,\dots, t_{d-1})\subseteq (T^{i_0}t_{i_0}+\dots + T^{i_l}t_{i_l}, \dots,T^{j_0}t_{j_0}+ \dots + T^{j_m}t_{j_m} )$, $\{i_0,\dots, i_l, \dots, j_0,\dots, j_m\}\subseteq d$,  and each $t_i$, $i<d$ appears in at most one combination in $\{ T^{i_0}t_{i_0}+\dots + T^{i_l}t_{i_l}, \dots,T^{j_0}t_{j_0}+ \dots + T^{j_m}t_{j_m}\}\subset \langle X \rangle$. Then $\phi$ has the property that for any $s,t\in \mathcal{F}$ it holds $f(s)=f(t)$ if and only if $\phi(s)=\phi(t)$.

   Assume that mixing is transitive. If $s=(s_0, \dots, s_{n_1})$ is an $n$-tuple and $w,v\in FIN_k$ are length one extensions of $s$, by $s\cup w$ we denote the $n+1$-tuple that extends $s$ by $w$. From now on for notational simplicity when we write $s\cup w$ we mean that $s\cup \{w\}=(s_0,\dots, s_{n_1},w)$ and when we write $s\cup w+v$ we mean the $n+1$-tuple $(s_0,\dots, s_{n_1}, w+v)$. For any $s\in \hat{ \mathcal{F}}\setminus \mathcal{F}$, $|s|=n$ and $X\in FIN_k^{[\infty]}$ with $[s,X]\neq \emptyset$, Theorem $1$ provides us with a $Y\in [s, X]$ and  $\phi_s$ so that for $s\cup w, s\cup v \in [s,Y]_{n+1}$, $Y$ mixes $ s\cup w$ with $ t\cup v$ if and only if $\phi_s(w)=\phi_t(v)$. 
   This is done by considering the equivalence relation $c:[s,X]_{n+1}\to \omega$, on $\langle X\rangle $, defined by $c(w)=c(v)$ if and only if $X$ mixes $s\cup w$ with $s\cup v$. By Proposition $1$, we can assume that for every $t \in \hat{ \mathcal{F}}\setminus \mathcal{F}\upharpoonright X$, there exists a staircase map $\phi_t$ which induces an equivalence relation on $[t,X]_{n+1}$.

Assume that $X$ mixes $s$ with $t$, $s,t\in \mathcal{A}X_n$, and consider the two-coloring $c': [t,X]_{n+1}\to 2$ defined by
\begin{equation*}
c'(t \cup w)=
\begin{cases}
1 & \text{ if }  Y \text{ mixes } t\cup w \text{ with } s\cup w \text{ and } \phi_t(w)=\phi_s(w),\\
0 & \text{ otherwise.}
\end{cases}
\end{equation*}

The fact that $\langle FIN_k^{[\infty]}, \leq, r\rangle$ is a topological Ramsey space gives us a $Z \leq Y$ where $c'\upharpoonright [t,Z]_{n+1}$ is constant. If the constant value is equal to one, then on $Z$ we have that for every $t\cup w \in [t,Z]_{n+1}$, Z mixes $t\cup w$ with $s\cup w$ and also $\phi_t(w)=\phi_s(w)$.

Let $c'\upharpoonright [t,Z]_{n+1}=0$. Then for every $t\cup w\in [t,Z]_{n+1}$ either there exists $s\cup v\in [s,Z]_{n+1}$ where $Z$ mixes $t\cup w$ with $s\cup v$ and $\phi_s(v)\neq \phi_t(w)$ or there is no $s\cup v\in [s,Z]_{n+1}$ that $Z$ mixes it with $t\cup w$. This allows to consider the two coloring: $c_1:[t,Z]_{n+1}\to 2$ defined by:
\begin{equation*}
c_1(t\cup w)=
\begin{cases}
1 & \text{ if there exists } s\cup v\in [s,Z]_{n+1},  Z \text{ mixes } t\cup w \text{ with } s\cup v,\\&  \phi_s(v)\neq \phi_t(w),\\
0 & \text{ otherwise.}
\end{cases}
\end{equation*}
 Once more there exists $Z_1\in [t,Z]$ so that $c_1\upharpoonright [t,Z_1]$ is constant. If the constant value is equal to zero, then $Z_1$ separates $s$ with $t$, a  contradiction to the assumption that $Z$ mixes $s$ with $t$. Suppose that $c_1\upharpoonright [t,Z_1]=1$. 
Then for every $t\cup w\in [t,Z_1]_{n+1}$ there exists $s\cup v \in [s,Z_1]_{n+1}$ so that $Z_1$ mixes $t\cup w$ with $s\cup v$ and $\phi_s(v)\neq \phi_t(w)$. Observe that it might be the case that $\phi_s=\phi_t$ and for every $t \cup w\in [t,Z_1]_{n+1}$, there exists $s\cup v\in [s,Z_1]_{n+1}$, $w\neq v$, so that $Z_1$ mixes $t\cup w$ with $s\cup v$ and even though $\phi_s=\phi_t$, $\phi_s(v)\neq\phi_t(w)$ cause $w\neq v$. The assumption that $\phi_s=\phi_t$ rules out the possibility of $Z_1$ mixing $t\cup w$ with $s\cup w$, for any $w\in \langle Z_1\rangle$. 

 Let $\phi_s=\cup_{i\in \alpha} f_i$, $\phi_t=\cup_{j\in \beta} g_j$, $\alpha, \beta \in k+1$, and $\mathcal{X} \subseteq \alpha, \mathcal{Y}\subset \beta$, $|\mathcal{X}|=|\mathcal{Y}|$, so that for every $i\in \mathcal{X}$ there $j\in \mathcal{Y}$ so that $f_i=g_j$. 
Let $\phi'_s=\cup_{i\in \alpha \setminus \mathcal{X}}f_i$ and $\phi'_t=\cup_{j\in\beta \setminus \mathcal{Y}}g_j$. 

For every $w\in \langle Z_1\rangle$ consider the finite sets: \\

$A^0_w=\{T^iw, i<k:  \phi'_s(T^iw+ v)=\phi'_s(T^iw+v') \text{ and } \phi'_t(T^iw+ v)=\phi'_t(T^iw+v') \text{ for all }v', v\in \langle Z_1/w\rangle\}$.\\

Notice that $A^0_w\neq \emptyset$ if and only if $\{f_i, g_j: i\in \alpha\setminus \mathcal{X}, j\in \beta \setminus \mathcal{Y}\} \subset  \mathcal{F}_{\min} \cup \mathcal{F}_{mid^0}$. Observe also that in the above definition $v,v'$ play not a role in the value of $\phi'_s, \phi'_t$, are needed only for $T^iw+v'\in FIN_k$. With $A^0_w$ we associate $F^0_w=\{i: T^iw\in A^0_w\}$.\\

Similarly if $\{f_i, g_j: i\in \alpha\setminus \mathcal{X},j\in  \beta \setminus \mathcal{Y}\} \subset  \mathcal{F}_{\max}  \cup \mathcal{F}_{mid^1}$ we define\\

 $A^1_w=\{T^iw, i<k:  \phi'_s( w'+ T^iw)=\phi'_s( v'+ T^iw) \text{ and } \phi'_t( w'+ T^iw)=\phi'_t( v'+ T^iw) \text{ for all }v', w'\in \langle Z_1/w\rangle\}$.\\
 
 With $A^1_w$ we associate $F^1_w=\{i: T^iw\in A^1_w\}$. Observe that given $\phi'_s$, $\phi'_t$, $F^0_{w}=F^0_v$ and $F^1_{w}=F^1_v$, for all $v$. Let\\

 $D^0_w=\{ (T^jw, T^iw)\in A^0_w\times A^0_w : Z_1 \text{ mixes } s \cup T^jw+ w'\text{ with }t \cup  T^iw+ v' \\ \text{ for all }v', w'\in \langle Z_1/w\rangle\}$. \\

With $D^0_w$ we associate the set $C^0_w=\{(j,i): (T^jw, T^iw)\in D^0_w\}$. 

Similarly we define\\

$D^1_w=\{ (T^jw, T^iw)\in A^1_w\times A^1_w : Z_1 \text{ mixes } s \cup w'+T^jw \text{ with } t \cup  v'+  T^iw \\ \text{ for all }v', w'\in \langle Z_1/w\rangle\}$. \\

With $D^1_w$ we associate the set $C^1_w=\{(j,i): (T^jw, T^iw)\in D^1_w\}$.

Notice that $C^i_w$, $i<2$, may contains pairs of the form $(j,i)$, where $j=i$. In the case that $\phi_s=\phi_t$ we have that both $C^0_w$ and $C^1_w$ do not contain elements of the form $(j,j)$, $j< k$. The rest is identical with the case that $\phi_s\neq \phi_t$.

Observe also that for every $w_0,w_1\in \langle Z_1 \rangle$, $i<k$ and $f\in  \mathcal{F}_{\max}  \cup \mathcal{F}_{mid^1} $, $f(T^iw_0+w_1)=f(w_1)$. In the case that $f \in \mathcal{F}_{\min} \cup \mathcal{F}_{mid^0}$, then $f(w_0 + T^iw_1)= f(w_0)$.

\begin{claim}Let $f,g \in \mathcal{F}_{\min} \cup \mathcal{F}_{mid^0}  \cup \mathcal{F}_{mid}$, $f\neq g$ and for $i< j$, it holds that $supp(f(w))\subset [\min_{i}(w),\min_{i+1}(w))$ and also $supp(g(w))\subset [\min_{j}(w),\min_{j+1}(w))$. Then for any $w \in \langle Z_1\rangle$ there always exists $T^{k-i}z$, $z<w$, so that $f(T^{k-i}z+w)\neq f(w)$ and $g(T^{k-i}z+w)=g(w)$.

Similarly if $f,g \in \mathcal{F}_{mid}\cup  \mathcal{F}_{\max}  \cup \mathcal{F}_{mid^1}$, $f\neq g$ and for $i< j$, it holds that $supp(f(w))\subset (\max_{i+1}(w),\max_{i}(w)]$ and also $supp(g(w))\subset (\max_{j+1}(w),\max_{j}(w)]$. For any $w\in \langle Z_1\rangle$ there always exists $T^{k-i}z$, $w<z$, so that $f(w+ T^{k-i}z)\neq f(w)$ and $g(w+ T^{k-i}z)=g(w)$.
\end{claim}

\begin{proof}Let $f,g \in \mathcal{F}_{\min} \cup \mathcal{F}_{mid^0}  \cup \mathcal{F}_{mid}$. Assume that $i<j$, pick any $z<w$ and consider $T^{k-i}z\in FIN_i$. Notice that $f(T^{k-i}z+w)\neq f(w)$ and $g(T^{k-i}z+w)=g(w)$. This is due to the fact that $\min_j(T^{k-i}z+w)=\min_j(w)$, $\min_{j+1}(T^{k-i}z+w)=\min_{j+1}(w)$ and for every $\theta^0_{j,l}$, $l\leq j$, it holds that $\theta^0_{j,l}(T^{k-i}z+w)=\theta^0_{j,l}(w)$.
 
 Let $f,g \in \mathcal{F}_{\max} \cup \mathcal{F}_{mid^1}  \cup \mathcal{F}_{mid}$ as in the claim.  Pick any $z>w$ and consider $T^{k-i}z\in FIN_i$. Notice that $f(w+T^{k-i}z)\neq f(w)$ and $g(w+T^{k-i}z)=g(w)$. This is due to the fact that $\max_j(w+T^{k-i}z)=\max_j(w)$, $\max_{j+1}(w+T^{k-i}z)=\max_{j+1}(w)$ and for every $\theta^1_{j,l}$, $l\leq j$, it holds that $\theta^1_{j,l}(w+T^{k-i}z)=\theta^1_{j,l}(w)$.

\end{proof}

Observe that on $\langle Z_1\rangle$ there is an one-to-one correspondence between equivalence classes induced by $\phi_s$ and $\phi_t$. Let once more $f,g \in \mathcal{F}_{\min} \cup \mathcal{F}_{mid^0}  \cup \mathcal{F}_{mid}$, $f\neq g$ and $supp(f(w))\subset [\min_{i}(w),\min_{i+1}(w))$, $supp(g(w))\subset [\min_{j}(w),\min_{j+1}(w))$. The case that $i=j$ occurs only in the following two possibilities. First if $f=\min_i$, $g=\theta^0_{i,h}$ and second if $f=\theta^0_{i,h'}$, $g=\theta^0_{i,h}$, $h,h'\leq i$. Let $f=\min_i$, $g=\theta^0_{i,h}$ for $h<i$ and consider $w=T^{k-i}w_0+w_1$. Suppose that $Z_1$ mixes $s\cup w$ with $t\cup w$. Pick $w_0<z<w_1$ and consider $v=T^{k-i}w_0+T^{k-h}z+ w_1$. Then $\phi_s(s\cup v)=\phi_s(s \cup w)$ and $\phi_t(t\cup v)\neq \phi_t(t\cup w)$. Therefore $Z_1$ mixes $s\cup w$ with $s\cup v$ and $t\cup w$. If $Z_1$ mixes $s\cup v$ with $t\cup v$, would imply that $Z_1$ mixes $t\cup w$ with $t\cup v$, but $\phi_t(t\cup w)\neq \phi_t(t\cup v)$.
As a result $Z_1$ separates $s\cup v$ with $t\cup v$. In the case that $h=i$ then $v= T^{k-i}w_0+T^{k-i}z+w_1$. Once more $Z_1$ separates $s\cup v$ with $t\cup v$.

Let now $f=\theta^0_{i,h'}$, $g=\theta^0_{i,h}$. The assumption that $f\neq g$ implies that $h\neq h'$. Suppose once more that $Z_1$ mixes $s\cup w$ with $t\cup w$, for $w= T^{k-i}w_0+w_1$. Pick $z$ so that $w_0<z<w_1$ and consider $v= T^{k-i}w_0+T^{k-h}z+w_1$, in the case that $h<h'$, and $v= T^{k-i}w_0+T^{k-h'}z+w_1$, in the case that $h'<h$. In the first case we have that $\phi_s(s\cup v)\neq \phi_s(s \cup w)$ and $\phi_t(t\cup v)= \phi_t(t\cup w)$. In the second case we have that $\phi_s(s\cup v)=\phi_s(s \cup w)$ and $\phi_t(t\cup v)\neq \phi_t(t\cup w)$. As a consequence $Z_1$ separates $s\cup v$ with $t\cup v$. Therefore given $f\in \phi_s$, $g\in \phi_t$, $f\neq g$, $f,g\in \mathcal{F}_{\min} \cup \mathcal{F}_{mid^0}  \cup \mathcal{F}_{mid}$ and $w$ so that $Z_1$ mixes $s\cup w$ with $t\cup w$, there exists $v$ that results from $w$ by addition, so that $Z_1$ separates $s\cup v$ with $t\cup v$. Let $f,g \in \mathcal{F}_{\max} \cup \mathcal{F}_{mid^1}  \cup \mathcal{F}_{mid}$ and $f\neq g$. By an identical argument we have that the above statement holds as well. \\

To avoid unnecessary length, from now onwards, we are going to assume that for every $f_n,g_m$ so that $f_n\in \phi'_s$ and $g_m\in \phi'_t$ it holds that it is not the case that both $supp(f_n(w)), supp(g_m(w)) \subset [\min_i(w), \min_{i+1}(w))$ or $supp(f_n(w)), supp(g_m(w)) \subset (\max_{i+1}(w), \max_{i}(w)]$, for the same $i\leq k$.

 Assume that for $n_0=\min \alpha \setminus \mathcal{X}$ one of the following holds. First case $supp(f_{n_0})\subset [\min_{i_0}(w),\min_{i_0+1}(w))$ and for all $j\in \beta\setminus \mathcal{Y}$, $supp(g_j(w))\subset [\min_{i}(w), \min_{i+1}(w))$, $i_0<i$ or $supp(g_j(w))\subset (\max_{i+1}(w), \max_{i}(w)]$ for any $i>i_0$. Second case $supp(f_{n_0})\subset (\max_{i_0+1}(w),\max_{i_0}(w)]$ and for all $j\in \beta\setminus \mathcal{Y}$, it holds that for $i>i_0$ $supp(g_j(w))\subset [\min_{i}(w), \min_{i+1}(w))$, or $supp(g_j(w))\subset (\max_{i+1}(w), \max_{i}(w)]$. Let $g_{m_0}$ be so that $m_0=\min \beta \setminus \mathcal{Y}$ and either $supp(g_{m_0}(w))\subset [\min_{i_m}(w), \min_{i_m+1}(w))$ or  $supp(g_{m_0}(w))\subset (\max_{i_m+1}(w), \max_{i_m}(w)]$. Therefore with $f_{n_0}$ we associate $i_0$ and with $g_{m_0}$ we associate $i_m$.

For $w\in \langle Z_1\rangle$ we introduce two more sets.\\
 
$B^0_w=\{T^iw: i<k, \text{ either } \exists i'\in\alpha \setminus \mathcal{X}, \text{ or } j'\in \beta\setminus \mathcal{Y} : f_{i'}(T^iw+w')\neq f_{i'}(w') \text{ or }\\ g_{j'}(T^iw+w')\neq g_{j'}(w'), \forall  w'\in \langle Z_1/w\rangle\},$

 and \\ $B^1_w=\{T^iw:0\neq  i<k, \text{ either }  \exists i'\in\alpha \setminus \mathcal{X}, \text{ or } j'\in \beta\setminus \mathcal{Y} : f_{i'}(w'+T^iw)\neq f_{i'}(w') \text{ or }\\ g_{j'}(w'+T^iw)\neq g_{j'}(w'), \forall  w'<w\}.$\\

 We claim the following.
 
\begin{claim}Let $w\in \langle Z_1\rangle$ and assume that for all $f_n \in \phi'_s$ and $g_m \in \phi'_t$ it holds that $f_n,g_m \in \mathcal{F}_{\min} \cup \mathcal{F}_{mid^0} $. As a result $A^0_w\neq \emptyset$. There exists $$\bar{w}=T^{j_l}z_l+\dots + T^{j_0}z_0+w,$$ $j_0< \dots < j_l$, so that for every $j,i\in F^0_{\bar{w}}$, not necessarily distinct, $Z_1$ separates $s \cup T^i\bar{w}+v$ with $t  \cup T^j \bar{w}+v'$, for any $v,v'\in \langle Z_1/\bar{w} \rangle$.
\end{claim}

\begin{proof} Let $w\in \langle Z_1 \rangle$ with $A^0_w\neq \emptyset$ be given and let $j=\max F^0_w$. Assume that $\phi_s\neq \phi_t$. Let $(j,j)\in C^0_w$.
Notice that $supp(f_{n_0}(w))<\max_k(w)$. From above by $i_0<k$ we denote the $supp(f_{n_0})(w)\subset [\min_{i_0}(w), \min_{i_0+1}(w))$. We have assumed also that $supp(g_{m_0}(w))\subset [\min_{i_m}(w), \min_{i_m}(w))$ and $i_m \geq i_0+1$. 
Pick any $z_0\in Z_1$ so that $z_0< w$ and set $w_0= T^{k-i_0-j}z_0+w$, where $T^{k-i_0-j}z_0\in FIN_{i_0+j}$. Notice that $w_0$ is a sos and $\phi_s( T^jw_0+w')\neq \phi_s( T^jw+w')$ and $\phi_t( T^jw_0+w')=\phi_t( T^jw+w')$.

Next consider $F^0_{w_0}$ and $C^0_{w_0}$. If $C^0_{w_0}\neq \emptyset$, then $(j,j)\notin C^0_{w_0}$. Let $j'=\max F^0_{w_0}$. If $(j',j')\in C^0_{w_0}$ repeat the above step to get $w_1= T^{k-i_0-j'}z_1+w_0$, $z_1<z_0$, so that $\phi_s( T^{j'}w_1+w')\neq \phi_s( T^{j'}w_0+w')$ and $\phi_t( T^{j'}w_0+w')=\phi_t( T^{j'}w_1+w')$.

Observe that $(j',j')\notin C^0_{w_1}$. Suppose now that both $(j_0,j_1), (j_1,j_0)\in C^0_{w_1}$, $j_0<j_1$. In other words $Z_1$ mixes $s\cup T^{j_0}w_1+w'$ with $t\cup T^{j_1}w_1+w'$ and also $s\cup T^{j_1}w_1+w'$ with $t\cup T^{j_0}w_1+w'$.
At this point we need the assumption it is not the case that $T^n\phi'_t \neq \phi'_s$, for $n<k$. We consider here the case where $\phi'_s=f_{i_0}$, $\phi'_t=g_{m_0}$, $i_m=i_0+1$ and $j_1=j_0+1$. This immediately implies that $i_0+j_1=i_m+j_0$.
 Assume also that $f_{n_0}(w)\in FIN_h$ and $g_{m_0}(w)\in FIN_{h'}$, where $h' \leq h$.

Pick $z_2<z_1$ and consider $w_2=T^{k-i_0-j_1}z_2+w_1$. 
If $Z_1$ still mixes $s\cup T^{j_1}w_2+w'$ with $t\cup T^{j_0}w_2+w'$, pick $z_3$ so that $z_2<z_3<w_1$ and consider $w_3= T^{k-i_0-j_1}z_2+ T^{k-h'-j_0}z_3+w_1$ 
Observe that $\phi_s(s\cup T^{j_1}w_3+w')= \phi_s(s\cup T^{j_1}w_2+w')$ and $\phi_t(t \cup T^{j_0}w_3+w')\neq \phi_t(t \cup  T^{j_0}w_2+w')$. In the case that $h'>h+1$ then $w_3= T^{k-i_0-j_1}z_2+ T^{k-h-j_1}z_3+w_1$ and we would have that $\phi_s(s\cup T^{j_1}w_3+w')\neq \phi_s(s\cup T^{j_1}w_2+w')$ and $\phi_t(t \cup T^{j_0}w_3+w')= \phi_t(t \cup  T^{j_0}w_2+w')$.
Therefore $Z_1$ separates $s\cup T^{j_1}w_3+w'$ with $t \cup T^{j_0}w_3+w'$. The reason that we need the assumption that $T^n\phi'_t\neq \phi'_s$, is that in the case of equality we will not be able by adding $z_2,z_3$, as we did just above, to separate $s\cup T^{j_1}w_3$ with $t \cup T^{j_0}w_3+w'$. This assumption cause not problem, see right after the end of this proof for a justification.

 If now $(j_0,j_1)\in C^0_{w_3}$ as well, pick $z_4<z_2$ and consider $w_4= T^{k-i_0-j_0}z_4+w_3$. Observe that $\phi_s(s\cup T^{j_0}w_4+w')\neq \phi_s(s\cup T^{j_0}w_3+w')$ and $\phi_t(t\cup T^{j_1}w_4+w') = \phi_t(t\cup T^{j_1}w_3+w')$. As a consequence $Z_1$ separates $s\cup T^{j_0}w_4+w'$ with $t\cup T^{j_1}w_4+w'$. All the other cases are dealt in an identical manner.
Proceed in this manner to get $\bar{w}$ that satisfies the conclusions of our claim.

In the case that $\phi_s=\phi_t$, as remarked above, the possibility of $(j,j)\in C^0_w$ does not occur for every $j<k$. Suppose that both $(j_0,j_1), (j_1,j_0)\in C^0_{w}$, $j_0<j_1$. In other words $Z_1$ mixes $s\cup T^{j_0}w+w'$ with $t\cup T^{j_1}w+w'$ and also $s\cup T^{j_1}w+w'$ with $t\cup T^{j_0}w+w'$ and $j_0<j_1$. For $z<w$, let $w_0=T^{k-i_0-j_0}z+w$. Notice that $\phi_s(T^{j_0}w_0+ w')\neq \phi_s(T^{j_0}w+w')$ and $\phi_t(T^{j_0}w_0+w')\neq \phi_t(T^{j_0}w+w')$. As a result both $(j_0,j_1)$ and $(j_1,j_0)$ are not in $C^0_{w_0}$. Proceed in this manner to get the $\bar{w}$ that satisfies the conclusions of our claim.
\end{proof}

Suppose that $T^jw,T^iw\in A^0_w$ and $i>j$. Suppose also that $T^{i-j}\phi'_t(w)=\phi'_s(w)$. In the case that $Z_1$ mixes $s\cup T^{i}w+v$ with $t\cup T^{j}w+v$, it is not possible to separate $s\cup T^{i}\bar{w}+v$ with $t\cup T^{j}\bar{w}+v$, for any $\bar{w}$ that results from $w$ by means of addition. 

The assumption that $T^{i-j}\phi'_t\neq \phi'_s$, causes not problem due to the fact that $\langle FIN^{[\infty]}_k,\leq,r\rangle$ is a topological Ramsey space. This reduces to the following sequence of colorings. Let $Z\in FIN_k^{[\infty]}$ and $\phi$ a staircase function that determines an equivalence relation on $\langle Z=(z_n)_{n\in \omega} \rangle$. In the case that every equivalence class of $\phi$ is infinite we proceed as follows. Let $i'<k$ be minimal so that  $\phi(z_0+T^{i'}v)= \phi(z_0+T^{i'}v')$ and for every $j<i'$, $\phi(z_0+T^{i'}v)\neq \phi(z_0+T^jv)$, where $v,v'\in \langle Z/z_0 \rangle$. Consider the coloring $c_0:\langle Z/z_0 \rangle \to 2$, defined by 

\begin{equation*}
c_0(w)=
\begin{cases}
1 & \text{ if } T^{i-j}\phi'_{z_0+T^{i'}w}= \phi'_{z_0}  \text{ and  } Z \text{ mixes } \\& z_0\cup T^{i}v+v' \text{ with } z_0+T^{i'}w \cup T^{j}v+v' ,  \\
0 & \text{ otherwise.}
\end{cases}
\end{equation*}

There exists $Z_0=(z^0_n)_{n\in \omega} \leq Z$ so that for $c_0\upharpoonright \langle Z_0 \rangle$ is constant. If the constant value is equal to one, observe that for all $z_0+T^{i'}z^0_0$, $z_0+T^{i'}z^0_0+T^{i'}v$ it does hold that $\phi'_{z_0+T^{i'}z^0_0}= \phi'_{z_0+T^{i'}z^0_0+T^{i'}v}$, for every $v\in \langle Z_0/z^0_0\rangle$. Repeat the above coloring for all $j'>i'$. Suppose that in theses cases the constant value is equal to zero.
Set $w_0=z_0+T^{i'}z^0_0$, and consider the finite set $A=\{ T^jw_0: j\leq k\}$. For every $v\in A$ consider the coloring $c_v:\langle Z_0/z^0_1 \rangle \to 2$ defined by 

\begin{equation*}
c_v(w)=
\begin{cases}
1 & \text{ if }T^{i-j}\phi'_{v+z^0_1}=\phi'_{v+z^0_1+T^{i'}w} \text{ and }  Z_0 \text{ mixes } \\ & v+z^0_1\cup T^{i}v' +v'' \text{ with } v+z^0_1+T^{i'}w\cup T^{j}v'+v'',\\
0 & \text{ otherwise.}
\end{cases}
\end{equation*}

After repeating that for all $v\in A$, $j'>i'$, we get $Z_1\leq Z_0$ and $w_1$ so that $\phi'_{v+w_1}= \phi'_{v+w_1+T^{i'}w'}$, for every $v\in A$ and  $w' \in \langle Z_1/w_1\rangle$, in the case that the constant values are equal to zero for the coloring corresponding to all $j'>i'$. Proceed in this manner to get $Z=(w_n)_{n\in \omega}$ where our assumption holds. Identical argument holds in the case that we color $T^{i-j}\phi'_{z_0}= \phi'_{z_0+T^{i'}w}$ and $T^{i-j}\phi'_{v+z^0_1+T^{i'}w}=\phi'_{v+z^0_1}$.

Observe that for every $k$, $\phi$ defines an equivalence relation, where each equivalence class has finitely many elements if and only if $\phi=\max_1$ or $\phi=\theta^1_{11}$. In this case we color as follows. Fix $w_0\in Z$ and consider the coloring $c_0: \langle Z/w_0\rangle \to 2$, defined by 

\begin{equation*}
c_2(w)=
\begin{cases}
1 & \text{ if } T^{i-j}\phi'_{w}=\phi'_{w_0+w} \text{ and } Z \text{ mixes }  \\ &w_0+w \cup T^{i}w'+v \text{ with } w\cup T^{j}w'+v, \\
0 & \text{ otherwise.}
\end{cases}
\end{equation*}

There exists $Z_0\leq Z$ so that $c_2\upharpoonright \langle Z_0/w_0\rangle $ is constant. If $c_0\upharpoonright \langle Z_0 /w_0\rangle =1$, notice that on $\langle Z_1/w_0\rangle$, $\phi(w)=\phi(v)$ implies that $\phi_{w}= \phi_v$. Then we will get $Z_1/w_0\leq Z_1$ so that our assumption holds. Identical in all other cases. \\

Notice that if  for all $f_n \in \phi'_s$ and $g_m \in \phi'_t$ it holds that $f_n,g_m  \in \mathcal{F}_{\max} \cup \mathcal{F}_{mid^1}$, by an identical argument with that of Claim $2$ we get that there always exists $\bar{w}=w+T^{j_0}z_0+\dots + T^{j_l}z_l$, $j_0< \dots < j_l$ so that for every $i,j\in F^1_{\bar{w}}$ $Z_1$ separates $s\cup w'+ T^i\bar{w}$ with $t\cup w'+ T^j\bar{w}$. 

\begin{observation}

In the case that both $A^0_w$ and $A^1_w$ are empty sets, $supp(f_{n_0}(w))<\min_k(w)$ and $Z_1$ mixes $s\cup w$ with $t \cup w$, then for any $z<w$, $Z_1$ separates $s\cup T^{k-i_0}z+w$ with $t\cup  T^{k-i_0}z+w$.

 Conversely if $supp(f_{n_0}(w))>\max_k(w)$ and and $Z_1$ mixes $s\cup w$ with $t \cup w$, then for any $z>w$, $Z_1$ separates $s\cup w+T^{k-i_0}z$ with $t\cup w+T^{k-i_0}z$. 
 \end{observation}
 
 For combinatorial purposes we consider strong systems of staircases a subset of the set of system of staircases. A $w\in FIN_k$ is a strong system of staircases if and only if $w= T^{k-1}w_1+T^{k-2}w_2+ \dots + Tw_{k-1}+w^0_k+w^1_k + Tw_{k+1}+\dots T^{k-2}w_{2k-2}+T^{k-1}w_{2k-1}$, for $w_1,\dots w_{2k-1}$ system of staircases. From now on we are considering strong system of staircases.

Next we claim the following.

\begin{claim}Let $w\in  Z_1/(s,t)$. There exists $Z_2\leq Z_1/w$ and $(T^{k-j_l}z_l)_{l\in \{1,\dots, k-1\}}$, $(T^{k-j^d_k}z^d_k)_{d\in m}$, where $(j_l)_l\in \{1,\dots, k \}$, $(j^d_k)_{d \in m}$ can be equal to $0$, so that for all $v\in \langle Z_2/(s,t) \rangle$, $T^jw,T^iw\in B^0_w$, not necessarily distinct, $Z_2$ separates $s\cup T^iw'+v$ with $t\cup T^jw'+v$. By $w'$ we denote $w'=T^{k-1}w_1+ T^{k-j_1}z_1+T^{k-2}w_2+ T^{k-j_2}z_2+ \dots +Tw_{k-1}+ T^{k-j_{k-1}}z_{k_1}+w^0_k+ T^{k-j^0_k}z^0_k + \dots +T^{k-j^{m-1}_k}z^{m-1}_k +w_k^1+ Tw_{k+1}+\dots T^{k-2}w_{2k-2}+T^{k-1}w_{2k-1}$.

\end{claim}

\begin{proof} Let $w\in  Z_1/(s,t)$, $T^{k-j}w,T^{k-i}w\in B^0_w$, so that $f\in \phi'_s$ witnesses that $T^{k-i}w\in B^0_w$ and $g\in \phi'_t$ witnesses that $T^{k-j}w\in B^0_w$. Consider the coloring $c_2:\langle Z_1/w\rangle \to 2$ defined as follows.
\begin{equation*}
c_2(v')=
\begin{cases}
1 & \text{ if } Z_1 \text { separates }s \cup T^{k-i}w+v' \text{ with } t\cup T^{k-j}w+v',\\
0 & \text{ otherwise.}
\end{cases}
\end{equation*}

By the fact that we are in a topological Ramsey space, we get $Z_2\leq Z_1/w$ so that $c_2\upharpoonright \langle Z_2\rangle$ is constant. If the constant value is equal to one, then the conclusions of our claim are satisfied for $w'=w$. 
Let the constant value be equal to zero. At this point we need the assumption that $T^{n}g_{m_0}w\neq f_{i_0}(w)$, for $n=i_m-i_0$. This assumption causes not a problem, see right after the end of this proof.
Let $f(w)\in FIN_h$, $g(w)=FIN_{h'}$ and assume that $f(w)=T^{j-i}g(w)$, $i<j$ and $f_{i_0}<f$. Add $T^{k-i_0-i}z$ to the right of $T^{k-i_0-i}w_{i_0+i}$, where $i_0<k$ is so that $supp(f_{i_0}(w))\subset [\min_{i_0}(w), \min_{i_0+1}(w))$, for $f_{i_0}\in \phi'_s$ as defined above, and observe that for all $v'\in \langle Z_2 \rangle$, $\phi_s(T^{k-i}w'+v')\neq \phi_s(T^{k-i}w+v')$ and $\phi_t(T^{k-j}w'+v') =\phi_t(T^{k-j}w+v')$. If now $f(w)\neq T^{j-i}g(w)$, $i<j$ and $h<h'$, add $T^{k-h-i}z$ to the right of $w^0_k$ and notice that $\phi_s(T^{k-i}w'+v')\neq  \phi_s(T^{k-i}w+v')$ and $\phi_t(T^{k-j}w'+v') = \phi_t(T^{k-j}w+v')$. Similarly in all the other cases.

Finally in the case that $h=h'$, assuming that $i<j$ add $T^{k-h-i}z$ to the right of $w^0_k$ and observe that $\phi_s(T^{k-i}w'+v')\neq  \phi_s(T^{k-i}w+v')$, $\phi_t(T^{k-j}w'+v') =\phi_t(T^{k-j}w+v')$. In the case that $j<i$, add $T^{k-h-j}z$ to the right of $w^0_k$ and observe that $\phi_s(T^{k-i}w'+v')= \phi_s(T^{k-i}w+v')$ and $\phi_t(T^{k-j}w'+v') \neq  \phi_t(T^{k-j}w+v')$.
Repeat this step to all possible such pairs $T^{k-j}w,T^{k-i}w\in B^0_w$ to get $w'$ that satisfies the conclusions of our claim.

\end{proof}
Suppose $f\in \phi'_s$ witnesses that $T^iw\in B^0_w$  and $g\in \phi'_t$ witnesses that $T^jw\in B^0_w$ and $i>j$. Suppose also that $T^{i-j}g(w)=f(w)$ and there is not $g'\in \phi'_t$ with $g'\neq g$, $g'<g$ and $f'\in \phi'_s$, $f'\neq f$, $f'<f$. In the case that $Z_1$ mixes $s\cup T^{i}w+v$ with $t\cup T^{j}w+v$, for all $v\in Z_1/w$, it is not possible to separate $s\cup T^{i}w'+v$ with $t\cup T^{j}w'+v$, for any $w'$ that results from $w$ by addition. This occurs in the case that either $T^{i-j}g(w)=f(w)$ and there is not $f'\in \phi'_s$, $f' \neq f$ where $f'<f$ and $g'\in \phi'_t$, $g'\neq g$ so that $g'<g$. This condition in our context of $s$ and $t$ as above, amounts to $T^ng_{m_0}\neq f_{i_0}$, for $n<k$.

The assumption that $T^{n}g_{m_0}(w)\neq f_{i_0}(w)$, causes not problem due to the fact that $\langle FIN^{[\infty]}_k,\leq,r\rangle$ is a topological Ramsey space. This reduces to the following sequence of colorings. Let $Z\in FIN_k^{[\infty]}$ and $\phi$ a staircase function that determines an equivalence relation on $\langle Z=(z_n)_{n\in \omega} \rangle$. In the case that every equivalence class of $\phi$ is infinite we proceed as follows. Let $i<k$ be minimal so that  $\phi(z_0+T^iv)= \phi(z_0+T^iv')$ and for every $j<i$, $\phi(z_0+T^iv)\neq \phi(z_0+T^jv)$, where $v,v'\in \langle Z/z_0 \rangle$. We are going to consider here the case where $n=i_m-i_0$. The case of any other $n$ is identical. Consider the coloring $c_0:\langle Z/z_0 \rangle \to 2$, defined by 

\begin{equation*}
c_0(w)=
\begin{cases}
1 & \text{ if } f_{i_0}= T^{i_m-i_0}g_{m_0} \text { where } f_{i_0}\in \phi'_{z_0}, g_{m_0}\in \phi'_{z_0+T^iw} \text{ and  } Z \text{ mixes } \\& z_0\cup T^{k-i_0}v+v' \text{ with } z_0+T^iw \cup T^{k-i_m}v+v' ,  \\
0 & \text{ otherwise.}
\end{cases}
\end{equation*}

There exists $Z_0=(z^0_n)_{n\in \omega} \leq Z$ so that for $c_0\upharpoonright \langle Z_0 \rangle$ is constant. If the constant value is equal to one, observe that for every $z_0+T^iz^0_0$, $z_0+T^iz^0_0+T^iv$ it does hold that $f^{z_0+T^iz^0_0}_{i_0}= f^{z_0+T^iz^0_0+T^iv}_{i_0}$, for all $v\in \langle Z_0/z^0_0\rangle$. In other words $i_0\in \mathcal{X}$, i.e. $f^{z_0+T^iz^0_0}_{i_0}\notin \phi'_{z_0+T^iz^0_0}$ and $f^{z_0+T^iz^0_0+T^iv}_{i_0}\notin \phi'_{z_0+T^iz^0_0+T^iv}$.
Notice that we can repeat this up to $k-3$ times, cause the new $f_{i'_0}^{z_0+T^iz^0_0}$ is so that $supp(f_{i'_0}^{z_0+T^iz^0_0})\subseteq (\min_{i'_0}(z_0+T^iz^0_0), \min_{i'_0+1}(z_0+T^iz^0_0))$ for $i'_0>i_0$. Therefore we will get to a block subsequence $Z$ that either the first alternative does not hold, or $\phi_{z_0+T^iz^0_0+T^iv}=\phi_{z_0+T^iz^0_0+T^iv'}$, for all $v,v'\in \langle Z/(z_0+T^iz^0_0)\rangle$. Repeat the above coloring for every $j>i$. Assuming that in all these colorings the constant value is $0$, set $w_0=z_0+T^iz^0_0$, and consider the finite set $A=\{ T^jw_0: j\leq k\}$. For every $v\in A$ consider the coloring $c_v:\langle Z_0/z^0_1 \rangle \to 2$ defined by 

\begin{equation*}
c_v(w)=
\begin{cases}
1 & \text{ if } f_{i_0}= T^{i_m-i_0}g_{m_0} \text { where } f_{i_0}\in \phi'_{v+z^0_1} \text{ and  } g_{m_0}\in \phi'_{v+z^0_1+T^iw}\text{ and }\\ &  Z_0 \text{ mixes } v+z^0_1\cup T^{k-i_0}v' +v'' \text{ with } v+z^0_1+T^iw\cup T^{k-i_m}v'+v'', \\ 
0 & \text{ otherwise.}
\end{cases}
\end{equation*}

After repeating that for all $v\in A$, as above, we get $Z_1\leq Z_0$ and $w_1$ so that $f^{v+w_1}_{i_0}= f^{v+w_1+T^iw'}_{i_0}$, for every $v\in A$ and  $w' \in \langle Z_1/w_1\rangle$. Proceed in this manner to get $Z=(w_n)_{n\in \omega}$ where our assumption holds. The above argument is identical in the case that $f_{i_0}\in  \phi'_{z_0+T^iw}$, $g_{m_0}\in\phi'_{z_0}$ and $f_{i_0}\in \phi'_{v+z^0_1+T^iw}$, $g_{m_0}\in \phi'_{v+z^0_1}$.

Observe that for every $k$, $\phi$ defines an equivalence relation, where each equivalence class has finitely many elements if and only if $\phi=\max_1$ or $\phi=\theta^1_{11}$. In this case we color as follows. Fix $w_0\in Z$ and consider the coloring $c_0: \langle Z/w_0\rangle \to 2$, defined by 

\begin{equation*}
c_0(w)=
\begin{cases}
1 & \text{ if } T^{i_m-i_0}g_{m_0}= f_{i_0},  f_{i_0}\in \phi_{w_0+w}, g_{m_0}\in \phi_w \text{ and } Z \text{ mixes } \\ & w_0+w \cup T^{k-i_0}w'+v \text{ with } w\cup T^{k-i_m}w'+v, \\
0 & \text{ otherwise.}
\end{cases}
\end{equation*}

There exists $Z_0\leq Z$ so that $c_0\upharpoonright \langle Z_0/w_0\rangle $ is constant. If $c_0\upharpoonright \langle Z_0 /w_0\rangle =1$, notice that on $\langle Z_1/w_0\rangle$, $\phi(w)=\phi(v)$ implies that $ f^w_{i_0}=f^v_{i_0}$, where $f^w_{i_0}\in \phi_{w}, f^v_{i_0}\in \phi_v$. As a result $\{ f^w_{i_0},f^v_{i_0} \} \subseteq \phi_w\cap \phi_v$. Observe that we can repeat this step up to $k-3$ times. Then we will get $Z'\leq Z_1$ so that our assumption holds.

Next we prove the following.

\begin{claim}Let $w\in \langle Z_1\rangle$ be given. There exists $Z_3\leq Z_1/w$ and $(T^{k-j_i}z_i)_{i\in \{1,\dots, k-1\}}$, where any $(j_i)_{i\in i\in \{1,\dots, k-1\}}$, can be equal to $0$, so that for all $v\in \langle Z_3/(s,t) \rangle$, $T^{k-j}v, T^{k-i}v\in B^1_v$, $Z_3$ separates $s\cup w'+T^{k-i}v$ with $t\cup w'+T^{k-j}v$, where $w'=T^{k-1}w_1+T^{k-2}w_2+  \dots +Tw_{k-1}+w^0_k +w^1_k+ T^{k-j_{k-1}}z_{k+1}+ Tw_{k-1}+  \dots + T^{k-j_2}z_2+ T^{k-2}w_{2k-2}+T^{k-1}w_{2k-1}+ T^{k-j_1}z_1$.

\end{claim}

\begin{proof} Let $w\in \langle Z_1\rangle$ be given. Let $T^{k-j}w, T^{k-i}w\in B^1_w$ where $f\in \phi'_s$ witnesses that $T^{k-i}w\in B^1_w$ and $g\in \phi'_t$ witnesses that $T^{k-j}w\in B^1_w$. Assume that $i<j$ and $f(w)\in FIN_h$, $g(w)\in FIN_{h'}$. Consider the coloring $c_3:\langle Z_1/w\rangle \to 2$ defined as follows.
\begin{equation*}
c_3(v)=
\begin{cases}
1 & \text{ if } Z_1 \text { separates }s \cup w+T^{k-i}v \text{ with } t\cup w+T^{k-j}v,\\
0 & \text{ otherwise.}
\end{cases}
\end{equation*}

By the fact that we are in a topological Ramsey space, we get $Z_3\leq Z_1/w$ so that $c_3\upharpoonright \langle Z_3\rangle$ is constant. If the constant value is equal to one, then the conclusions of our claim are satisfied for $w'=w$. If the constant value is equal to zero we proceed as follows. We have assumed that $i<j$ which implies that $k-i>k-j$. Add $T^{k-j}z_j$ to the left of $T^{k-j}w_j$ and notice that $\phi_s(w'+T^{k-i}v')= \phi_s(w+T^{k-i}v')$ and $\phi_t(w'+T^{k-j}v') \neq \phi_t(w+T^{k-j}jv')$.

 Repeat that for all possible pairs $T^{k-j}v, T^{k-i}v\in B^1_v$, to get $w'$ that satisfies the conclusions of our claim.

\end{proof}

Given $s,t$ as above and $Z_1$ so that $c_1 \upharpoonright [t,Z_1]_{n+1}=1$. Pick a $z'\in Z_1/(s,t)$. If $Z_1$ mixes $s\cup z'$ with $t\cup z'$ and $supp(f_{i_0})<\min_k$, then for $z_0<z$, we get $w=T^{i_0}z_0+ z'$ so that $\phi_s(s\cup z')\neq \phi_s(s \cup w)$ and $\phi_t(t \cup z')=\phi_t(t \cup w)$. 
Observe that $Z_1$ separates $s\cup w$ with $t \cup w$. This is due to the fact the equivalence class, induced by $\phi_s$, of $z'$ is mixed with that of $z'$, induced by $\phi_t$. Notice that since  $\phi_t(t \cup z')=\phi_t(t \cup w)$, $w$ and $z'$ are in the same equivalence class, induced by $\phi_t$. If now $Z_1$ mixes $s\cup w$ with $t\cup w$ it would imply that $w$ is the same class with $z'$ contradicting that $\phi_s(s\cup z')\neq \phi_s(s\cup T^iz+z')$.
We are going to construct a $Z'\in [t, Z_1]$ so that $Z'$ separates $s$ with $t$. This gives us a contradiction, which would imply that the possibility $c'\upharpoonright [t,Z]_{n+1}=0$ does not occur.

Consider $w\in Z_1$. In the case that $A^0_w\neq \emptyset$ by Claim $2$ we get $\bar{w}$, so that for every $i,j\in F^0_w$, $Z_1$ separates $s\cup T^i\bar{w}+w'$ with $t\cup T^j\bar{w}+w'$, for all $ w'\in \langle Z_1/\bar{w} \rangle$. By Claim $3$ there exists $Z_2\leq Z_1$ and $w'$, so that for every $i,j$, $T^iw,T^jw\in B^0_w$, $Z_2$ separates $s\cup T^jw'+v$ with $t \cup T^iw'+v'$, for all $v,v'\in \langle Z_2/w'\rangle$. Notice that it might be the case that $C^0_{w'}\neq \emptyset$. In this case by an addition we get $\tilde{w}$ so that $C^0_{\tilde{w}}=\emptyset$. This addition does not ruin the conclusions of Claim $3$ cause in the case that $Z_2$ separates $s\cup T^iw'+T^iw+v$ with $t\cup T^jw+v$ and it also separates $s\cup T^iw''+T^iw'+T^iw+v$ with $t\cup T^jw+v$. Notice also that Claim $3$ contributes only on at most $|\phi'_s|$ many levels of the staircase. If it contributes on all $\{1, \dots ,k-1\}$ levels, then $|A^0_w|=1$.
Set $w_0=\tilde{w}$. For $s\cup w_0$ there exists $t \cup v'\in [t,Z_1/w_0]_{n+1}$ so that $Z_1$ mixes them. Conversely $t \cup w_0$ is mixed with $s\cup v$. Let $Z_2$ be the reduct of $Z_1$ that avoids $v,v'$, i.e. $v,v'\notin \mathcal{A}Z_1$.

Suppose that we have constructed $ w_0, \dots w_{n-1}$ and $Z_n$ with the property that $Z_n$ separates $s \cup w$ with $t \cup v$ for $w,v\in \langle w_0, \dots , w_{n-1}\rangle$. Pick $w'_n \in  Z_n/w_{n-1}$. By Claim $2$ we get $\bar{w_n}$ so that for every $i,j\in F^0_{\bar{w_n}}$, $Z_n$ separates $s\cup T^i\bar{w_n}+v'$ with $t \cup T^j\bar{w_n}+v'$, for all $v'\in \langle Z_n/\bar{w_n} \rangle$. Next by Claim $3$ we get $w'_n$ so that $Z_n$ separates $s\cup T^iw'_n+ v' $ with $t\cup T^jw'_n+ v'$ for all $T^jw'_n, T^iw'_n\in B^0_{w'_n}$, $v'\in \langle Z_n/w'_n \rangle$. Once more let $\tilde{w_n}$ be so that $C^0_{\tilde{w_n}}=\emptyset$. Set $w_n=\tilde{w_n}$ and let $Z_{n+1}\leq Z_n$ so that $Z_{n+1}$ avoids all $v,v'\in \langle Z_n/w_n\rangle$ where $Z_n$ mixes $s\cup w_n$ with $t\cup v$, $t\cup w_n$ with $s\cup v'$ and $t\cup T^jw+w_n$ with $s\cup v'$, $s\cup T^jw+w_n$ with $t\cup v$, for $T^jw\in B^0_w$, $w\in \langle w_0, \dots , w_{n-1}\rangle$. In this way we built $Z=(w_n)_{n\in \omega}$ that separates $s$ with $t$.

The case where $A^1_w\neq \emptyset$ is identical with the above, except that we are using Claim $2$ and then Claim $4$, instead of Claim $2$ and then Claim $3$. 
Now suppose that $\phi_s$ and $\phi_t$ are so that for all $w$ both $A^0_w=A^1_w=\emptyset$. As a result we consider only $B^0_w\neq \emptyset$ and $B^1_w\neq \emptyset$. We start by picking $w$. Observation $1$ gives us $\bar{w}$ so that $Z_1$ separates $s\cup \bar{w}$ with $t\cup \bar{w}$. Then apply Claim $3$ to get $w'_0$ and $Z'_1\leq Z_1$, so that for every $T^jw'_0,T^jw'_0\in B^0_{w'_0}$, $Z'_1$ separates $s\cup T^iw'_0+v$ with $t\cup T^jw'_0+v$, for all $v\in \langle Z'_1/w'_0\rangle$. Next by Claim $4$ we get $w''_0$ and $Z''_1\leq Z'_1$ so that $Z''_1$ separates $s\cup w''_0+T^iv$ with $t\cup w''_0+T^jv$, for all $v\in \langle Z''_1/\tilde{w_1}\rangle$, $T^jv,T^iv\in B^1_v$. As above observe that it might be the case that $Z''_1$ mixes $s\cup w''_0$ with $t\cup w''_0$. In this case add $T^{k-i_0}z$ to $w''_0$ so that the resulting $\tilde{w_0}$ has the property that $Z''_1$ separates $s\cup \tilde{w_0}$ with $t\cup \tilde{w_0}$. Set $w_0=\tilde{w_0}$. Let $Z_2\leq Z_1$ that avoids $v,v'\in \langle Z''_1\rangle$ where $Z''_1$ mixes $s\cup w_0$ with $t\cup v$ and $t\cup w_0$ with $s\cup v'$.

Suppose that we have constructed $ w_0, \dots w_{n-1}$ and $Z_n$ with the property that $Z_n$ separates $s \cup w$ with $t \cup v$ for $w,v\in \langle w_0, \dots , w_{n-1}\rangle$. Pick $w\in Z_n/w_{n-1}$. By Observation $1$ we get $\bar{w}$ so that $Z_n$ separates $s\cup \bar{w}$ with $t\cup \bar{w}$. By Claim $3$ we get $w'_n$ and $Z'_n$ so that $Z'_n$ separates $s\cup T^iw'_n+w'$ with $t\cup T^jw'_n+w'$, for all $w'\in \langle Z'_n/w'_n\rangle$, $T^jw'_n,T^iw'_n\in B^0_{w'_n}$. Then by Claim $4$ we get $w''_n$ and $Z''_n$ so that $Z''_n$ separates $s\cup w''_n+T^iv$ with $t\cup w''_n+T^jv$ for all $v\in \langle Z''_n/w'_n\rangle$, $T^jv,T^iv\in B^1_v$. As we noticed above, it might be the case that $Z''_n$ mixes $s\cup w''_n$ with $t\cup w''_n$. Add $T^{k-i_0}z$ to $w''_n$, so that the resulting $\tilde{w_n}$ has the property they $Z''_n$ separates $s\cup \tilde{w_n}$ with $t\cup \tilde{w_n}$. Set $w_n=\tilde{w_n}$. Let $Z_{n+1}\leq Z''_n$ so that it avoids all $v,v'\in \langle Z''_n/ w_n\rangle$ where $Z''_n$ mixes $s\cup w_n$ with $t\cup v$, $t\cup w_n$ with $s\cup v'$ and $s\cup T^iw+w_n$ with $t\cup v$, $t\cup T^iw+w_n$ with $s\cup v$ as well as $s\cup w+T^iw_n$ with $t\cup v$ and conversely, for $w\in \langle w_0, \dots , w_{n-1}\rangle$, $T^iw\in B^0_w$, $T^iw_n\in B^1_{w_n}$. Let $Z=(w_n)_{n\in \omega}$. Then $Z$ separates $s$ with $t$.

As a result we have that $c'\upharpoonright [t,Z]=1$. Therefore if $X$ mixes $s$ with $t$, then on $\langle Z \rangle$ we have that $\phi_s=\phi_t$. Proposition $2$ tells us that this is in fact an if and only if statement. 
Now we are in a position to define the inner map $\phi$ as follows \begin{equation} \phi(s)=\bigcup_{s'\sqsubseteq s}\phi_{s'}(s(|s'|)). \end{equation} Next we show that for every $s,t\in \mathcal{F}$ it holds that $f(s)=f(t)$ if and only if $\phi(s)=\phi(t)$.

  \begin{lemma}The following are true for all $Y\leq X$.
  \begin{enumerate}
 \item{} Let $s,t \in \hat{ \mathcal{F}}\setminus \mathcal{F}$. If $\phi_s\neq \emptyset$ and $\phi_t=\emptyset$, there exists $w\in [s, X]_{|s|+1}$ so that $X$ mixes $t$ with $s\cup w$ with at most one equivalence class of $ [s, X]_{|s|+1}$.
  
  \item{} If $X$ separates $s$ with $t$, then its separates $s\cup w$ with $t\cup v$ for all $w\in  [s, X]_{|s|+1}$ and $v\in  [t, X]_{|t|+1}$.
  
  \item{} If $s\sqsubset t$, $s,t\in \mathcal{F}$  and $\phi(s)=\phi(t)$, then $X$ mixes $s$ with $t$.
  \end{enumerate}
    \end{lemma}
    
    \begin{proof}
    Suppose that $X$ mixes $t$ with $s\cup w$, and also $t$ with $s\cup v$, and $\phi_s( w)\neq \phi_s( v)$. By Lemma $1$, we get that $X$ mixes $s\cup w$ with $s\cup v$, a contradiction.
    
    Suppose that $X$ separates $s$ with $t$ and there exists $w\in [s, X]_{|s|+1}$ and $v\in  [t, X]_{|t|+1}$ so that $X$ mixes $s\cup w$ with $t\cup v$. This means that for every $Y\leq X$ there exists $w'\in \mathcal{F}_{s\cup w}\upharpoonright Y$ and $v'\in \mathcal{F}_{t\cup v}\upharpoonright Y$ so that $f(s\cup w \cup w')=f (t \cup v \cup v')$. Therefore $Y$ mixes $s\cup w\cup w'$ with $t\cup v \cup v'$, a contradiction to our assumption that $X$ separates $s$ with $w$. 
    
    Suppose now that  $s,t\in \mathcal{F}$, $s\sqsubset t$ and $\phi(s)=\phi(t)$. This means that for all $j\in [|s|,|t|]$, $\phi_{r_j(t)}=\emptyset$, which, by an induction on $n= [|s|,|t|]$, implies that $s$ gets mixed by $X$ with all the extensions of $r_{|s|}(t)$. In particular $X$ mixes $s$ with $t$.
    
    \end{proof}
  
  \begin{lemma}
  For $s,t\in \hat{\mathcal{F}}$ if $\phi(s)=\phi(t)$, then $X$ mixes $s$ with $t$. In particular if $s,t\in \mathcal{F}$ and $\phi (s)=\phi (t)$, then $c(s)=c(t)$.
  \end{lemma}
  
  \begin{proof} The proof is by induction on $l< \max (depth_X(s), depth_X(t))$. For $l=0$, $s\cap r_0(U)=t\cap r_0(X)=\emptyset$, so $X$ mixes $s\cap r_0(X)$ with $t\cap r_0(X)$. Assume that $X$ mixes $s\cap r_{l-1}(X)$ with $t\cap r_{l-1}(X)$ and consider $s\cap r_l(X)$ and $t\cap r_l(X)$. If $s\cap r_l(X)\neq \emptyset$ and $t\cap r_l(X)=\emptyset$, then we must have that $\phi_{s\cap r_l(X)}=\emptyset$. This implies that $s\cap r_l(X)$ is mixed with $s\cap r_{l-1}(X)$ which is mixed with $t\cap r_l(X)$. Therefore $s\cap r_l(X)$ is mixed with $t\cap r_l(X)$. Similarly if $s\cap r_l(X)= \emptyset$ and $t\cap r_l(X)\neq \emptyset$. 
  \end{proof}

\begin{lemma}
For $s,t\in \hat{\mathcal{F}}$, $s\neq t$ it doesn't hold that $\phi(s)\sqsubseteq \phi(t)$.
  \end{lemma}

\begin{proof}Suppose that there are $s,t\in \hat{\mathcal{F}}$ with $\phi(s)\sqsubset \phi(t)$. Let $j<\omega$ be so that $\phi(s)=\phi(r_j(t))$. There is at least one $i\in \omega, i>j$ so that $\phi_{r_i(t)}\neq \emptyset$. Assume that $\phi_{r_{j}(t)}\neq \emptyset$, which implies that $X$ mixes $s$ with $r_{j}(t)\cup v$, for some $v$ that belongs to the equivalence relation on $[r_j(t),X]_{j+1}$ induced by $\phi_{r_j(t)}$. But then consider a reduct $Y\leq X$ that avoids $v$. Then $Y$ separates $s$ with $r_j(t)$, a contradiction.

\end{proof}

    \begin{lemma}
    For $s,t\in \mathcal{F}$, if  $c(s)=c(t)$, then $\phi(s)=\phi(t)$.
  \end{lemma}                   
                       
 \begin{proof}
 Let $s,t\in \mathcal{F}$ with $c(s)=c(t)$. Then for every $l<\max (depth_X(s), depth_X(t))$, $X$ mixes $s\cap X(l)$ with $t\cap X(l)$. We show by induction that for all such an $l$ it holds that $\phi(s\cap X(l))=\phi(t\cap X(l))$. For $l=0$ $s\cap X(0)=t\cap X(0)=\emptyset$. Assume that $\phi(s\cap X(l-1))=\phi(t\cap X(l-1))$ and consider $s\cap X(l)$ and $t\cap X(l)$. Assume that $s\cap X(l)\neq \emptyset$ and $t\cap X(l)=\emptyset$. If $\phi_{s\cap X(l-1)}\neq \emptyset$ then we must have that $t\neq t\cap X(l)$. If that was the case it will contradict the above lemma since $\phi(t)=\phi(t\cap X(l))\sqsubset \phi(s)$. Notice that $X$ mixes $s\cap X(l-1)$ with $t\cap X(l-1)$, since $\phi(s\cap X(l-1))=\phi(t\cap X(l-1))$. This implies that $\phi_{s\cap X(l-1)}=\phi_{t\cap X(l-1)}$. But $s\cap X(l)\neq \emptyset$, $\phi_{s\cap X(l-1)}\neq \emptyset$ and $t\cap X(l)=\emptyset$, a contradiction.
 \end{proof}    
 
 Obviously $\phi$ is an inner mapping, i.e. for every $t\in \mathcal{F}$, $\phi(t)\subseteq t$. Lemma $4$ shows that is not the case that $\phi(s)\sqsubseteq \phi(t)$, for $s\neq t$. The fact that $X$ mixes $s$ with $t$ if and only if for every $s\cup w\in [s,X]_{|s|+1}$ then for $t\cup w \in [t,U]_{|t|+1}$, $X$ mixes $s\cup w$ with $t\cup w$ and $\phi_s(w)=\phi_t(v)$ implies that $\phi(s)\nsubseteq \phi(t)$. 
 
 Next we prove that $\phi$ is maximal among all other mappings representing $f:\mathcal{F}\upharpoonright X\to \omega$.
 
 \begin{lemma} Suppose $Y\leq X$ and there is another $\phi'$ map, satisfying that for all $t_0,t_1\in \mathcal{F}\upharpoonright Y$ $f(t_0)=f(t_1)$ if and only if $\phi'(t_0)=\phi'(t_1)$. Then there exists $Z\leq Y$ so that for every $s\in \mathcal{F}\upharpoonright Z$ $\phi'(s) \subseteq \phi(s)$.
 \end{lemma}        
 
 \begin{proof} By Corollary $2$ and Proposition $1$, we can assume that $\phi'$ has the form of Definition $2$. To see this, for any $t\in \mathcal{F}\upharpoonright X$, $i<|t|$, by Corollary $2$ there exists $X'\in [r_{i}(t),X]$ so that for every $s\supset r_{i}(t)$, $\phi'(s)\cap s(i)=g(s(i))$, for $g\in \mathcal{G}$, as in Definition $4$. If this is done for an arbitrary $t\in \mathcal{F}\upharpoonright X$, by Proposition $1$ we can assume that it holds for every $t\in \mathcal{F}\upharpoonright X$.

 Pick $t\in \mathcal{F}\upharpoonright Y$. Let $n=|t|$. For $i<n$ consider both $\phi'_{r_i(t)}$ and $\phi_{r_i(t)}$. Consider the two coloring $c':[r_i(t),Y]_{i+1}\to 2$, defined by
 \begin{equation*}
c'(v)=
\begin{cases}
1 & \text{ if } \phi'_{r_i(t)}(v)\subseteq \phi_{r_i(t)}(v),\\
0 & \text{ otherwise.}
\end{cases}
\end{equation*}

There exists $Z'\leq [r_i(t),Y]$ so that $c'\upharpoonright [r_i(t),Z']_{i+1}$ is constant and equal to one. Observe that we can only have for every extension $v$ of $r_i(t)$, so that $r_i(t)\cup v\in [r_i(t),Z']_{i+1}$, $\phi'_{r_i(t)}(v)\subseteq \phi_{r_i(t)}(v)$. This is due to the fact that both $\phi'$ and $\phi$ witness the same $f\upharpoonright (\mathcal{F}\upharpoonright \langle Y \rangle)$. By Proposition $1$, there exists $Z\leq Z'$ that satisfies the conclusions of our proposition.
 
 \end{proof}

 This looks after the transitive case.

\section{Non transitive mixing}

We consider now the case when mixing from Definition $5$ is not transitive. For every $k$, and a coloring $f:\mathcal{F}\to \omega$, where $\mathcal{F}$ is a front, our definition of mixing says that $X\in FIN_k^{[\infty]}$ mixes $s$ with $t$, where $s,t\in \hat{\mathcal{F}}$, if and only if for every $Y\leq X$ there are $\bar{s}, \bar{t}\in \mathcal{F}\upharpoonright Y$, so that $s\sqsubseteq \bar{s}$, $t\sqsubseteq \bar{t}$ and $f(\bar{s})=f(\bar{t})$. As the example in the last section illustrates, it might be the case that $\bar{s}\setminus s, \bar{t}\setminus t$ are not in $Y/(s,t)$. Therefore the mixing is not actually taking place on the "tail" of $X$. As Lemma $1$ demonstrates when the mixing is taking place on the "tail" of $X$, then transitivity holds. This necessitates the introduction of the notion of weak mixing as follows.

\begin{definition} Let $f:\mathcal{F}\to \omega$, where $\mathcal{F}$ is a front on $X\in FIN_k^{[\infty]}$. For $s,t\in \hat{\mathcal{F}}$ we say that $X$ weakly mixes $s$ with $t$, if $depth_X(s)<depth_X(t)$ and there exist $w_s^t\in t\setminus s$ so that for every $Y\leq X$, $[s,Y]\neq \emptyset$, $[t,Y]\neq \emptyset$, there exists $\bar{s},\bar{t}\in \mathcal{F}\upharpoonright Y$, $t\sqsubseteq \bar{t}$, $s\cup w_s^t+v\sqsubseteq \bar{s}$, for some $v\in Y/w_s^t$, such that $f(\bar{s})=f(\bar{t})$.

\end{definition}

Notice that if $X$ weakly mixes $s$ with $t$, then $X$ mixes $s$ with $t$.
The very first instance of the above definition is the following case. Let $\mathcal{F}=\mathcal{A}X_2$, $s\in \mathcal{A}X_1$ and $\phi_s\neq \emptyset$. Assume also that  for a fixed $w'\in X/s$ $\phi_{ w'}=\emptyset$ and $f(s\cup w')=f(w'\cup v)$. As a consequence $f(s\cup w')=f(w' \cup v)$, for every $v\in \langle X/w' \rangle$. 
Observe that any $Y\leq X$, so that both $[s,Y]\neq \emptyset$ and $[w',Y]\neq \emptyset$, $Y$ mixes $s$ with $w'$, but $\phi_s\neq \phi_{w'}$. In fact in this case the mixing of $s$ and $ w'$ is already decided from the $depth_X( w')$, and is irrelevant to the "tail" of $X$. According to Definition $6$ $X$ weakly mixes $s$ with $w'$ and $w_s^{w'}=w'$, $v=\emptyset$. Observe that the above happens cause $\phi_{w'}=\emptyset$. If $\phi_{w'}\neq \emptyset$ then a reduct $X'\leq X$ would separate $s$ with $w'$. Let $c:\mathcal{A}X_1\to 2$ defined by 
\begin{equation*}
c(w)=
\begin{cases}
1 & \text{ if } \phi_w=\emptyset ,\\
0 & \text{ otherwise.}
\end{cases}
\end{equation*}

There exists $Y\leq X$ so that $c\upharpoonright \mathcal{A}Y_1$ is constant. On $Y$ instances of $s$ and $w'$ as above do not occur. Let $F$ be a front of the form $\mathcal{A}X_n$, for $n\in \omega$. Then by induction on $n$ and a coloring as above, we can also assume that if $X$ weakly mixes $s$ and $t$ then for every $t\cup v\in [t,X]_{|t|+1}$ there exists $s\cup w_s^t+v'\in [s,X]_{|s|+1}$, $v'\neq \emptyset$,  so that $X$ mixes $t\cup v$ with $s\cup w_s^t+v'$.

 Next we observe the following. 

\begin{claim} If $s$ is weakly mixed by $X$ with $t$ and $t$ is also weakly mixed with $p$, then $X$ weakly mixes $s$ with $p$ as well and $w_s^t\subseteq w_s^p$. Similarly in the case that $X$ weakly mixes $s$ with $t$, $X$ mixes, not weakly, $t$ with $p$, then $X$ weakly mixes $s$ with $p$ and $w_s^t\subseteq w_s^p$. 

\end{claim}
\begin{proof}In the first case, by definition, we have that $depth_X(s)< depth_X(t) < depth_X(p)$. Suppose $X$ weakly mixes $s$ with $t$, then for $w^t_s\in t\setminus s$ we have that for every $Y\leq X$, compatible with both $s,t$, there exists $\bar{s},\bar{t}\in \mathcal{F}\upharpoonright Y$ where $s\cup w^t_s\sqsubseteq \bar{s}$, $t\sqsubseteq \bar{t}$ and $f(\bar{s})=f(\bar{t})$. Consider the coloring $c:[t,X]_{|t|+1}\to 2$ defined by 
\begin{equation*}
c(t\cup v)=
\begin{cases}
1 & \text{ if } X \text{ mixes } t\cup v \text{ with } s\cup w^t_s + v', v'\in \langle X/w_s^t \rangle ,\\
0 & \text{ otherwise.}
\end{cases}
\end{equation*}

There is $X_0\leq X$ so that $c\upharpoonright [t,X_0]_{|t|+1}=1$. By a similar coloring we get a further reduct $X_1\leq X_0$, $w^p_t\in p\setminus t$, so that every $p\cup v\in [p,X_1]_{|p|+1}$ is mixed by $X_1$ with 
$t\cup w_t + v'$, for $v'\in \langle X_1/w_t^p\rangle$. Then $X_1$ mixes $s$ with $p$. If $X_1$ mixes, but doesn't mixes weakly, $p$ with $s$, then we would have that for every $v\in \langle X_1/p \rangle $ there exists $v', v'' \langle X_1/p \rangle $ so that $X_1$ mixes $p\cup v$ with $s\cup v'$ and $p\cup v$ with $t\cup w^p_t+v''$. As a result $X$ mixes $s\cup v'$ with $t\cup w^p_t+v''$, a contradiction. Therefore there exists $w^p_s\in p\setminus s$ where for every $v\in \langle X_1/p \rangle $, there exists $v', v'' \in \langle X_1/p \rangle $, so that $X_1$ mixes $p\cup v$ with $s\cup w^p_s+v'$ and $p\cup v$ with $t\cup w^p_t+v''$. This implies that $w^t_s\subseteq w^p_s$ and as a result $w^t_s\subseteq p\setminus s$. 

Now in the case that $X$ mixes, but not weakly, $t$ with $p$, then $X$ mixes $s$ with $p$. If $s$ and $p$ are mixed by $X$, not weakly, then we can assume that for every $v\in \langle X_1/p \rangle $, there exists $v',v''\langle X_1/p \rangle $ so that $X$ mixes $p\cup v$ with both $t\cup v''$ and $s\cup v'$, contradicting that $s$, $t$ are weakly mixed. Therefore there exists $w^p_s\in p\setminus s$ so that for every $v\in \langle X_1/p \rangle $, there exists $v', v'' \in \langle X_1/p \rangle $, where $X_1$ mixes $p\cup v$ with $s\cup w^p_s+v'$ and $p\cup v$ with $t\cup v''$. As a result $X_1$ mixes $s\cup w^p_s+v'$ with $t\cup v''$ as well. This implies that $w_s^t\subseteq w_s^p\subseteq p\setminus s$.
\end{proof}
                 
       The above claim shows the following. Let $s,t\in \hat{ \mathcal{F}}\setminus \mathcal{F}$, $|s|=|t|=n$, so that are weakly mixed by $X$, i.e. for every $Y\leq X$ there are extensions $s\cup w_s^t \sqsubset \bar{s}$ and $t\sqsubseteq \bar{t}$, $\bar{s},\bar{t}\in \mathcal{F}\upharpoonright Y$, so that $f(\bar{s})=f(\bar{t})$. Any $p\in \hat{ \mathcal{F}}\setminus \mathcal{F}$, $|p|=n$, where there exists $p\sqsubseteq \bar{p}\in \mathcal{F}\upharpoonright X$ with $f(\bar{s})=f(\bar{t})=f(\bar{p})$, is so that $w_s^t\subseteq p\setminus s$ as well. Observe that $w_s^t$ is in the second part of the tuple $s\cup v$, but on the first part of $\bar{p}$ and $\bar{t}$. As a consequence $f$ factors through addition. Therefore transitivity is ruined if and only if $f$ factors through addition.

  Define $\phi$ in this case as follows. 
  
  \begin{equation} \phi(t_0,\dots, t_{d-1})=  (g_0(T^{i_0}t_{i_0}+\dots + T^{i_{l_0}}t_{i_{l_0}}), \dots, g_{n-1}(T^{j_0}t_{j_0}+ \dots + T^{j_{l_{n-1}}}t_{j_{l_{n-1}}})), \end{equation}
  
  where $(g_i)_{i\in n}\in \mathcal{F}$, $n<d-1$ and we require that if $l_h=0$, for some $h<n$, then $i_{0}=0$. Obviously $\phi(t)\subseteq (T^{i_0}t_{i_0}+\dots + T^{i_l}t_{i_{l_0}}, \dots, T^{j_0}t_{j_0}+ \dots + T^{j_m}t_{j_{l_{m-1}}})$. Notice that the arguments in Lemmas $2--6$ hold identically with this $\phi$ as well.

  Then $(2)$ covers the case where the mixing is not transitive. As a result the proof of Theorem $2$ is complete.


\begin{thebibliography}{5}
\bibitem[Do-To]{Do-To}  A Ramsey-classification theorem and its application in the Tukey theory of ultrafilters, Transactions of AMS volume 366, number 3 1659-1684.
\end{thebibliography}
     \end{document}